\documentclass[12pt]{amsart}
\usepackage{amsmath,amsthm,amsfonts,amssymb,amscd}
\usepackage[all,cmtip]{xy}
\usepackage{comment}
\usepackage{color}

\begin{document}

\newcommand{\A}{{\mathbb A}}
\newcommand{\B}{{\mathbb B}}
\newcommand{\C}{{\mathbb C}}
\newcommand{\N}{{\mathbb N}}
\newcommand{\Q}{{\mathbb Q}}
\newcommand{\Z}{{\mathbb Z}}
\renewcommand{\P}{{\mathbb P}}
\renewcommand{\O}{{\mathcal O}}
\newcommand{\R}{{\mathbb R}}
\newcommand{\rc}{\subset}
\newcommand{\rank}{\mathop{rank}}
\newcommand{\trace}{\mathop{tr}}
\newcommand{\dimc}{\mathop{dim}_{\C}}
\newcommand{\Lie}{\mathop{Lie}}
\newcommand{\Spec}{\mathop{Spec}}
\newcommand{\Auto}{\mathop{{\rm Aut}_{\mathcal O}}}
\newcommand{\alg}[1]{{\mathbf #1}}
\newtheorem{lemma}{Lemma}[section]
\newtheorem*{definition}{Definition}
\newtheorem*{claim}{Claim}
\newtheorem{corollary}{Corollary}
\newtheorem*{Conjecture}{Conjecture}
\newtheorem*{SpecAss}{Special Assumptions}
\newtheorem{example}{Example}
\newtheorem*{remark}{Remark}
\newtheorem*{observation}{Observation}
\newtheorem*{fact}{Fact}
\newtheorem*{remarks}{Remarks}
\newtheorem{proposition}[lemma]{Proposition}
\newtheorem{theorem}[lemma]{Theorem}
\numberwithin{equation}{section}
\def\labelenumi{\rm(\roman{enumi})}
\title[Fundamental group and pluridifferentials ]{Fundamental group and pluridifferentials on compact K\"ahler manifolds}
\author {Yohan Brunebarbe \& Fr\'ed\'eric Campana}
\begin{abstract}
A compact K\"ahler manifold $X$ is shown to be simply-connected if its `symmetric cotangent algebra' is trivial. Conjecturally, such a manifold should even be rationally connected. The relative version is also shown: a proper surjective connected holomorphic map $f:X\to S$ between connected manifolds induces an isomorphism of fundamental groups if its smooth fibres are as above, and if $X$ is K\"ahler.\end{abstract}
%
\address{%
Yohan Brunebarbe \\
EPFL\\
Lausanne\\
Suisse
}
\address{%
Fr\'ed\'eric Campana\\
Institut Elie Cartan\\
Universit\'e de Lorraine\\
France\\
et Institut Universitaire de France
}

\maketitle
\section{Introduction}

We shall show:

\begin{theorem}\label{main theorem}
Let $X$ be a connected compact K\"ahler manifold. Suppose that for all $p \geq 1$ and $k \geq 1$ there is no non-zero global section of the sheaf $S^k \Omega_X^p$. Then $X$ is simply connected\footnote{By a theorem of Kodaira, any $X$ as above is actually projective.}. 
\end{theorem}

This theorem refines a former result of \cite{Campana95} with the very same statement, but with $\otimes ^k \Omega_X^p$ in place of $S^k \Omega_X^p$. The proof  of \ref{main theorem} is obtained by refining the proof of \cite{Campana95}, which rests on $L^2$-methods \`a la Poincar\'e-Atiyah-Gromov.

The `uniruledness conjecture' below implies easily (see \S \ref{abuconj}) that $X$ should, in fact, be rationally connected, hence simply-connected, by \cite{Campana91}. Theorem \ref{main theorem} above permits to bypass this conjecture, as far as the fundamental group is concerned. It is usually quite easy to verify the vanishings of all $S^k \Omega_X^p$, while constructing sufficiently many rational curves requires the characteristic $p>0$ methods introduced by S. Mori, no characteristic zero proof being presently known.

The weaker assumption that $H^0(X, S^k \Omega_X^1)=\{0\}$ for every $k  \geq 1$ implies (see \cite{BKT}) that all linear representations of the fundamental group $\pi_1(X) \to GL_n(K)$, $K$ a field, have finite image. This raises the question of whether the condition $H^0(X, S^k \Omega_X^1)=\{0\}$ for every $k  \geq 1$ might imply that $\pi_1(X)$ is finite, instead of trivial. Enriques surfaces (examples of general type also exist) indeed show that simple-connectedness may then fail\footnote{Hopf surfaces $X$ have $H^0(X, S^k \Omega_X^p)=\{0\},\forall k>1,p>1$, showing that the K\"ahler assumption cannot be removed in \ref{main theorem}, since $\pi_1(X) \cong \Z$.}.

In contrast to the condition $H^0(X, S^k \Omega_X^p)=\{0\}$ for every $k \geq 1$ and $p \geq 1$, the condition $H^0(X, S^k \Omega_X^1)=\{0\}$ for every $k \geq 1$ does not seem however to have an even conjectural geometric interpretation in the frame of bimeromorphic classification of compact K\"ahler manifolds.

\

The theorem \ref{main theorem} above has a relative version, shown in section \S \ref{corollary} below:

\begin{corollary}\label{relative version} Let $f: X\to S$ be a proper holomorphic map with connected fibres between connected complex manifolds. Assume\footnote{These hypothesis should imply that $f$ is projective, locally above $S$.} that $X$ admits a K\"ahler metric, and that $f_*(S^k(\Omega^p_{X/S}))=0$ for every $k \geq 1$ and $p \geq 1$. Then $f_*:\pi_1(X)\to \pi_1(S)$ is an isomorphism of groups.
\end{corollary} 

Note that the conclusion of corollary \ref{relative version} may fail for a projective morphism $f:X\to S$ with smooth fibres simply-connected, because of the possible presence of multiple fibres. Consider indeed an Enriques surface $Y$ and its $K3$ universal cover $Y'\to Y=Y'/\Z_2$. Let $C\to \Bbb P^1=C/\Z_2$ be the $2$-sheeted cover defined by a hyperelliptic curve $C$. Now let $X\to S:=\Bbb P^1$ be deduced from the first projection $X':=C\times Y'\to C$ by taking the equivariant quotient by the involution $u\times v$ acting freely on $X'$, $u$ and $v$ being the involutions on $Y'$ and $C$ respectively deduced from the $\Z_2$ covers above. Here $S=\Bbb P^1$ is simply connected although $\pi_1(X)$ is a $\Z_2$ extension of $\pi_1(C)$ and the smooth fibres of $f$ are simply-connected.


\section{Proof of theorem \ref{main theorem}}

As in \cite{Campana95}, the proof goes in two steps: show first that $\pi_1(X)$ is finite (this is the main step, established below), and then show, using Serre's covering trick, that $\pi_1(X)$ is in fact trivial. 

We start by establishing this second step. Let $\pi : X' \to X$ be a finite Galois \'etale cover of $X$ of group $G$ and degree $d$. The Euler characteristic of the structural sheaf of $X$ 
\begin{equation*}
\chi (X , \mathcal{O}_X ) := \sum_{i = 0}^{\dim X} (-1)^i \cdot h^{i}(X, \mathcal{O}_X )
\end{equation*}

is equal to $1$, since by Serre's duality $h^{i}(X, \mathcal{O}_X) = h^0(X , \Omega_X^{i})$, and the latter is zero for $i \neq 0$ by hypothesis.

Now, if $\omega \in H^0(X', \Omega_{X'}^{i}) $, the product of the $g^* \omega$ for $g\in G$ defines an element of $H^0(X', S^d \Omega_{X'}^{i}) $ invariant by the action of $G$. We obtain in this way a global section of $S^d \Omega_X^{i} $, which is non zero  if $\omega$ is non zero. Thus it follows from the hypothesis that we must also have $\chi (X' , \mathcal{O}_{X'} ) = 1$. 

From the multiplicativity of the Euler characteristic (see lemma \ref{multiplicativity Euler characteristic} below), we get: $$ 1 =  \chi (X' , \mathcal{O}_{X'} ) = d \cdot \chi (X , \mathcal{O}_X ),$$
and $d$ is then necessarly equal to $1$.

\begin{lemma}\label{multiplicativity Euler characteristic} Let $X' \to X$ be a finite \'etale covering of degree $d$ of compact complex analytic spaces. Then
\[  \chi (X' , \mathcal{O}_{X'} ) = d \cdot \chi (X , \mathcal{O}_X ). \]
\end{lemma}
\begin{proof}
When $X$ is projective, an elementary proof due to Kleiman is given in \cite{Lazarsfeld1}, exemple 1.1.30. In general, it is an easy consequence of the theorem of Riemann-Roch-Hirzebruch, which is proved in \cite{Levy} for compact complex analytic spaces\footnote{We shall only need the case when $X$ is a divisor with normal crossings in a complex K\"ahler manifold in the proof of corollary \ref{relative version}.}. 
\end{proof}

\paragraph{}
To complete the proof of theorem \ref{main theorem}, we need to show that the fundamental group of $X$ is finite. Equivalently, we have to show the

\begin{theorem}\label{infinite fundamental group}
Let $X$ be a connected compact K\"ahler manifold with infinite fundamental group. Then there exists $p \geq 1$ and $k \geq 1$ such that $H^0(X, S^k \Omega_X^p) \neq \{0\}$.
\end{theorem}

\begin{proof}

Let $p : \tilde{X} \to X$ be the universal cover of $X$. The fundamental group $\Gamma := \pi_1(X)$ acts on $\tilde{X}$. The choice of a K\"ahler metric on $X$ induces a complete K\"ahler metric on $\tilde{X}$. Denote by $\mathcal{H}_{(2)}^k (\tilde{X})$ the Hilbert space of $L^2$-harmonic complex-valued forms of degree $k$ on $\tilde{X}$. Recall that a $p$-form $\alpha$ is called harmonic if $\Delta \alpha = 0$, where $\Delta := d \circ d^* + d^* \circ d$ and $d^* := -* \circ \,  d \circ *$. Moreover, a $L^2$ $p$-form $\alpha$ is harmonic if and only if $d \alpha = 0$ and $d^* \alpha = 0$ (the metric being complete), if and only if $\overline{\partial} \alpha = 0$ and $\overline{\partial}^*\alpha = 0$ (the metric being complete and K\"ahler), see \cite{Gromov}.

The decomposition in types gives rise to a orthogonal sum 
\[ \mathcal{H}_{(2)}^k (\tilde{X}) = \bigoplus_{p+q = k} \mathcal{H}_{(2)}^{p,q} (\tilde{X}).\]
The space $\mathcal{H}_{(2)}^{p,0} (\tilde{X}) $ consists of the $L^2$-holomorphic $p$-forms on $\tilde{X}$.

The Hilbert spaces $\mathcal{H}_{(2)}^{p,q}  (\tilde{X}) $ might be infinite dimensional. Nevertheless, using the isometric action of $\Gamma$ on them, one can associate to them a non-negative real number $\dim_{\Gamma}(\mathcal{H}_{(2)}^{p,q}  (\tilde{X}) )$ (cf. \cite{Atiyah}). This number is zero if and only if $ \mathcal{H}_{(2)}^{p,q}  (\tilde{X}) = \{0\} $. 

By Atiyah's $L^2$-index theorem (cf. \cite{Atiyah, Gromov}), we know that 
\begin{equation*}\label{index}
  \chi (X , \mathcal{O}_X )  =  \chi_{(2)} (\tilde{X} , \mathcal{O}_{\tilde{X}} ) := \sum_{q=0}^{\dim X} (-1)^q \cdot \dim_{\Gamma}(\mathcal{H}_{(2)}^{0,q}(\tilde{X}) )
\end{equation*}

Observe that there are no non-zero $L^2$-holomorphic functions on $\tilde{X}$. Indeed, the metric being complete, any harmonic function is closed, hence locally constant. By hypothesis $\tilde{X}$ is non-compact, and any constant $L^2$ function has to be zero.   \\
Let us distinguish two cases. Suppose first that $   \chi (X , \mathcal{O}_X )  = 0 $. Since $ \dim H^0(X, \mathcal{O}_X ) = 1$, Hodge symmetry shows that $H^0(X,  \Omega_X^p) \neq \{0\}$, for some (odd) $p \geq 1$, and the theorem is proved in this case. If, now, $\chi (X , \mathcal{O}_X )\neq 0 $, it follows from the discussion above that there exists $p \geq 1$ such that $\mathcal{H}_{(2)}^{0,p}(\tilde{X}) \neq \{0 \} $. By conjugation $\mathcal{H}_{(2)}^{p,0}(\tilde{X}) \neq \{0 \} $, hence we get a non-zero $L^2$-holomorphic $p$-form for some $p \geq 1$.

The rest of the proof consists, following \cite{Gromov}, in constructing from this $L^2$ section a non-zero $\Gamma$-invariant section of some $S^k \Omega_{\tilde{X}}^p$. This can be done using a construction which goes back to Poincar\'e, that we now describe in a general setting.

\paragraph{}
Let $M$ be a complex manifold and $E$ be a holomorphic vector bundle on $M$. Let $\Gamma$ be a countable discrete group acting on $M$ and suppose that the action of $\Gamma$ lifts to an action on $E$. Let $h_E$ be a $\Gamma$-invariant continuous hermitian metric on $E$. Let $\Phi : \Bbb P(E) \to M$ denote the projective bundle of hyperplanes in $E$ and $  \mathcal{O}_{E}(1) \to \Bbb P(E) $ be the tautological line bundle endowed with the induced hermitian metric $h_L$. By functoriality the group $\Gamma $ acts on $\Bbb P(E)$ and $\mathcal{O}_E(1)$, and all the maps considered above are $\Gamma$-equivariant. As $\Phi_* (\mathcal{O}_E(k)) = S^k E$ for all $k \geq 1$ (where $\mathcal{O}_E(k)   $ denotes the line bundle $\mathcal{O}_E(1)^{\otimes k}$), there is a $\Gamma$-equivariant identification between the space of holomorphic sections $H^0(\Bbb P(E), \mathcal{O}_E(k)) = H^0(M, S^k E)$ under which $L^q$ holomorphic sections are identified for all $q \geq 1$.

To any $L^1$ holomorphic section $s$ of $E$ we can associate a $\Gamma$-invariant section of $S^k E$ for all $k \geq 1$ (the so-called Poincar\'e series) as follows :

\begin{equation*}
P_k(s)(x) := \sum_{\gamma \in \Gamma} \gamma^{*} s^k (\gamma \cdot x)
\end{equation*}

As $s$ is $L^1$, this series converges absolutely to a $\Gamma$-invariant holomorphic section of $S^k E$.

Moreover, if $s$ is not the zero section, then $P_k(s)$ is non-zero for infinitely many $k \geq 1$. Indeed, the precededing construction shows that we need only to consider the case where $E$ is a line bundle. The assertion is then a consequence of the following lemma.

\begin{lemma}\label{L1} (See Lemma 3.2.A from \cite{Gromov})
Let $\{a_i\}$ be an $l^1$-sequence of complex numbers, not all zero. Then there are infinitely many $k \geq 1$ such that $\sum_{i} a_i^k \neq 0$. 
\end{lemma}

Now recall that in the case where $   \chi (X , \mathcal{O}_X )\neq 0$, we showed the existence of a non-zero $L^2$ section $s$ of $ \Omega_{\tilde{X}}^p$ for some $ p>0$. If we see $s$ as a section of the tautological line bundle $\mathcal{O}_{\Omega_{\tilde{X}}^p}(1)$ on the projectified bundle of $\Omega_{\tilde{X}}^p$, then $s^{\otimes k}$ is a non-zero $L^1$ section of $\mathcal{O}_{\Omega_{\tilde{X}}^p}(1)$ for any $k\geq 2$. Applying the averaging construction just described to $s^{\otimes 2}$, we get a non-zero $\Gamma$-invariant section of some $\mathcal{O}_{\Omega_{\tilde{X}}^p}(2k)$, giving a non-zero section of $S^{2k} \Omega_{X}^p$, as claimed. This concludes the proof\footnote{We thank C. Mourougane for observing  that in our first version, our construction appeared to give a section of $S^k(S^2(\Omega_{X}^p))$, instead of $S^{2k}(\Omega_{X}^p)$.}. \end{proof}

\begin{remark} For any compact connected K\"ahler manifold $X$ with infinite fundamental group, let $P(X)$ (resp. $ P_{(2)}(X)$) be the set of integers $p$ such that $ H^0(X,S^k\Omega_X^p)\neq \{0\} $ for some $k>0$ (resp. such that $ H^0_{(2)}(X',S^k\Omega_{X'}^p)\neq \{0\}$ for some $k>0$ and some infinite connected \'etale cover $X'$ of $X$). The arguments above show that $P_{(2)}(X)\subset P(X)$. Complex tori show that this inclusion can be strict.
\end{remark}

\section{A criterion for rational connectedness.}\label{abuconj}

Recall the following consequence of the `Abundance Conjecture'

\begin{Conjecture}\label{cur}(`uniruledness' conjecture) Let $X$ be a connected compact K\"ahler manifold. Then  $X$ is uniruled (i.e. covered by rational curves) if and only if $H^0(X, K_X^{\otimes k})=\{0\}$ for all $k>0$.
\end{Conjecture} 

Consider also the following conjecture:

\begin{Conjecture}\label{crc} Let $X$ be a connected compact K\"ahler manifold. Then $X$ is rationally connected (i.e. any two generic points are joined by some rational curve) if and only if $H^0(X,S^k\Omega^p_X)=0$, for every $k>0$ and $p>0$\footnote{A weaker form, usually attributed to D. Mumford, claims the same conclusion assuming that $H^0(X, (\Omega_X^1)^{\otimes k})=\{0\}$ for all $k>0$.}.
\end{Conjecture} 

In \cite{CDP} a weaker form of Conjecture \ref{crc} is established: $X$ is rationally connected if $H^0(X,S^k\Omega^p_X\otimes A)=0$, for every $k>k(A)$, every $p>0$, and some ample line bundle $A$ on $X$.

For both conjectures, the ``only if" part is easy. The second conjecture implies theorem \ref{main theorem} above, since rationally connected manifolds are simply connected \cite{Campana91}.

Let us show that the first conjecture implies the second. First, a K\"ahler manifold $X$ as in the second conjecture has $h^{2,0}(X)=0$, so it is projective algebraic by Kodaira's projectivity criterion. Now consider the so-called `rational quotient' $r_X:X\dashrightarrow R$ (constructed in \cite{Campana92} and in \cite{KMM}, where it is called the `MRC'-fibration), which has rationally connected fibres and non-uniruled base $R$ (by \cite{GHS}). Assuming that $r:= \dim(R)>0$, we get a contradiction, since by the first conjecture there exists a non-zero $s\in H^0(R,K_R^{\otimes k})$, for some $k>0$, which lifts to $X$ as a non-zero section of $H^0(X,S^k\Omega_X^r)$. Thus $r=0$ and $X$ is rationally connected.

\begin{remark}
For any compact connected K\"ahler manifold, let $r^-(X):= \max\{p\geq 0 \vert \exists k>0 , H^0(X,S^k\Omega_X^p)\neq \{0\}\}.$
Let $r(X):= \dim(R), R$ as above. The preceding arguments show that $r(X)\geq r^-(X)$, and the uniruledness conjecture is equivalent to the equality: $r(X)=r^-(X)$.
\end{remark}

\section{Proof of corollary \ref{relative version}}\label{corollary}

The corollary \ref{relative version} is an easy consequence of the theorem \ref{main theorem} and the following, the proof and statement of which are inspired by \cite{Kollar93}, theorem 5.2:

\begin{theorem}\label{relative} Let $f:X\to S$ be a proper holomorphic map with connected fibres between connected complex manifolds. Assume that $X$ admits a K\"ahler metric and that there exists a smooth fibre $X_s$ of $f$ which is simply-connected and satisfies $H^p(X_s, \mathcal{O}_{X_s}) = 0$ for all $p>0$. Then $f_*:\pi_1(X)\to \pi_1(S)$ is an isomorphism of groups.
\end{theorem}

\begin{proof}

First observe that all the smooth fibres $X_s$ of $f$ are simply-connected and satisfy $H^p(X_s, \mathcal{O}_{X_s}) = 0$ for all $ p>0$. Indeed, the restriction of $f$ to its smooth locus $S^{o} \subset S$ is topologically a locally trivial fiber bundle by Ehresmann's lemma, and the dimension of $H^p(X_s, \mathcal{O}_{X_s})$ is locally constant for $s \in S^{o} $, as follows from the theory of variations of Hodge structures.

Let us first consider the following special case: $X$ is a connected complex K\"ahler manifold, $f : X \rightarrow \Delta$ is a proper holomorphic map with connected fibres, smooth outside $0 \in \Delta$. Recall that in this situation $X_0$ is a retract of $X$. We have to show that the fundamental group of $X$ (which is isomorphic to $\pi_1(X_0)$) is trivial. By blowing-up $X$, one can ensure that $X_0$ has only simple normal crossings (i.e. the irreducible components of the corresponding reduced divisor are smooth and meet transversally); this does not change the fundamental group of $X$.
By (\cite{Kollar93}, lemma 5.2.2) the fundamental group of $X$ is finite cyclic, say of order $d$. Let $\pi : \tilde{X} \to X$ be a universal cover of $X$ and $g: \tilde{X} \to \Delta$ be the Stein factorization of $f \circ \pi$ so that:  
\begin{align*}
\xymatrix{
\tilde{X} \ar[d]^{\pi} \ar[r]^{g} & \Delta \ar[d]^{t \mapsto t^d}\\
X \ar[r]^{f} & \Delta}
\end{align*}

The fibre $\tilde{X}_t$ of $g$ at any $t \neq 0$ is isomorphic to $X_{t^d}$, hence $H^p(\tilde{X}_t, \mathcal{O}_{\tilde{X}_t}) = H^p(X_{t^d}, \mathcal{O}_{X_{t^d}}) = 0$ for $t \neq 0$ and $p > 0$, and  the sheaves $R^pg_*\mathcal O_{\tilde{X}}$ are generically zero for all $p>0$. Being torsion-free (see \cite{Steenbrink77}, theorem 2.11\footnote{in this reference the morphism is supposed projective but the same proof works for a proper morphism assuming that the total space admits a K\"ahler metric. See also \cite{Peters-Steenbrink}, corollary 11.18.}), they are in fact zero on $\Delta$. Using Leray's spectral sequence, this implies that $ H^p(\tilde{X}, \mathcal{O}_{\tilde{X}}) = H^p( \Delta ,g_*\mathcal O_{\tilde{X}}) = 0 $ for $p> 0$. Applying the lemma \ref{Steenbrink} below, it follows that $H^p(\tilde{X}_0^{red}, \mathcal{O}_{\tilde{X}_0^{red}}) = 0$ for all $p>0$, hence $\chi(\tilde{X}_0^{red}, \mathcal{O}_{\tilde{X}_0^{red}}) = 1$.
By multiplicativity of the holomorphic Euler characteristic in finite \'etale cover (see lemma \ref{multiplicativity Euler characteristic}), $d = 1$ and $X$ is simply-connected.

\begin{lemma}\label{Steenbrink}(Steenbrink, see \cite{Steenbrink83} lemma 2.14 and \cite{Kollar93} lemma 5.2.3)
Let $X$ be a complex K\"ahler manifold and let $D \subset X$ be a reduced divisor such that $D$ as a complex space is proper and has normal crossing only. Assume moreover that $D$ is topologically a retract of $X$. Then the restriction maps $H^p(X, \mathcal{O}_{X})  \to H^p(D, \mathcal{O}_{D}) $ are surjective for all $p \geq 0$.
\end{lemma}

\begin{proof}
Fix a $p \geq 0$. Since $D$ is topologically a retract of $X$, the map $H^p(X, \C)  \to H^p(D, \C) $ is an isomorphism. On the other hand, as $D$ is a union of compact K\"ahler manifolds crossing transversally, $H^p(D, \C)$ admits a canonical mixed Hodge structure (see \cite{Griffiths-Schmid} section 4) whose Hodge filtration $H^p(D, \C) = F^0 H^p(D, \C) \supseteq F^1 H^p(D, \C) \supseteq \cdots $ satisfies $Gr^0_F H^p(D, \C) \cong H^p(D,  \mathcal{O}_{D}) $, see \cite{Steenbrink83} section (1.5). It follows that the map $H^p(D, \C) \to H^p(D, \mathcal{O}_{D})$ is surjective.
The following commutative diagram
\begin{align*}
\xymatrix{
H^p(X, \C) \ar@{=}[d] \ar[r] & H^p(X, \mathcal{O}_{X})  \ar[d]\\
H^p(D, \C) \ar@{->>}[r] & H^p(D, \mathcal{O}_{D})}
\end{align*}
shows that $H^p(X, \mathcal{O}_{X})  \to H^p(D, \mathcal{O}_{D}) $ is surjective.
\end{proof}

We now  reduce the general case to this special case. First, because of the following diagram, theorem \ref{relative} for $f$ follows from the corresponding statement for the restriction of $f$ to an open $U:=S-T$, if the codimension in $S$ of $T$, Zariski closed in $S$, is at least $2$: 
\begin{align*}
\xymatrix{
\pi_1(f^{-1}(U)) \ar@{->>}[d] \ar[r]^{f_*} & \pi_1(U) \ar@{=}[d]\\
\pi_1(X) \ar[r]^{f_*} & \pi_1(S)}
\end{align*}

On the other hand, any $s \in S$ admits a contractible neighborhood $U$ in $S$ such that $f^{-1}(U)$ is homeomorphic to $ U \times f^{-1}(s)$  (see for example \cite{Lê-Teissier}). From this, one easily sees that the theorem \ref{relative} for $f : X \rightarrow S$ follows if all fibres $X_s$ are simply-connected, at least for $s$ outside a codimension $\geq 2$ closed subvariety by the preceding observation. 

Let $D \subset S$ be the proper closed subset of points $s$ for which $X_s$ is not smooth. By removing a codimension $\geq 2$ subvariety of $S$, one can assume that $D$ is a smooth divisor in $S$. Now, an easy application of Sard's lemma shows that for $s \in D$ outside a proper subvariety $Z \subset D$, there exists a small disk $\Delta_s$ crossing $D$ transversally at $s$ such that $f^{-1} (\Delta_s)$ is smooth. For any $s \in D- Z$, the restriction of $f$ to $\Delta_s$ satisfies the assumptions of the special case of theorem \ref{relative} that we showed above, hence $\pi_1(X_s) = \pi_1(  f^{-1} (\Delta_s) ) = \{1\}$.

\end{proof}

Let us now explain how the theorems \ref{main theorem} and \ref{relative} imply the corollary \ref{relative version}. First observe that for fixed $k>0$ and $p>0$, the dimension of $H^0(X_s , (S^k\Omega^p_{X/S})_{|X_s})$ is constant on a non empty Zariski open subset of $S$, and this dimension has to be zero by the flat base change theorem. It follows that $H^0(X_s , S^k \Omega^p_{X_s}) = 0$ for all $k>0$ and $p>0$ for a general smooth fibre $X_s$ of $f$. By theorem \ref{main theorem} this implies that a general smooth fibre of $f$ is simply connected; hence every smooth fibre is simply connected.
The same argument shows that, in particular, for all $p>0$, $h^0(X_s, \Omega^p_{X_s})=0$, for $s\in S$ generic, and so: $h^p(X_s, \mathcal{O}_{X_s})=0$ by Hodge symmetry. We can thus apply theorem \ref{relative} to conclude the proof of corollary \ref{relative version}.


\end{document}